\documentclass[11pt]{amsart}
\usepackage{graphicx,color}
\usepackage{amssymb}
\usepackage{epstopdf, xypic}
\usepackage{booktabs}

\newtheorem{thm}{Theorem}[section]
\newtheorem{prop}[thm]{Proposition}

\newtheorem{lem}[thm]{Lemma}
\newtheorem{cor}[thm]{Corollary}
\newtheorem{conj}[thm]{Conjecture}
\newtheorem{qst}[thm]{Question}
\newtheorem{ex}[thm]{Example}

\newtheorem*{introcor}{Theorem (Corollary~\ref{maincor})}

\theoremstyle{definition}

\theoremstyle{remark}
\newtheorem{rmk}[thm]{Remark}

\newcommand{\N}{\mathbb N}

\renewcommand{\phi}{\varphi}
\newcommand{\im}{\mathop{\mathrm{Im}}\nolimits}
\renewcommand{\ker}{\mathop{\mathrm{Ker}}\nolimits}

\newcommand{\onto}{\twoheadrightarrow}

\renewcommand{\i}[1]{\mathfrak{#1}}

\newcommand{\m}{\i{m}}

\renewcommand{\a}{\mathbf{a}}

\newcommand{\depth}{\mathop{\mathrm{depth}}\nolimits}

\newcommand{\ann}{\mathop{\mathrm{Ann}}\nolimits}

\newcommand{\codim}{\mathop{\mathrm{codim}}\nolimits}

\newcommand{\pd}{\mathop{\mathrm{pd}}\nolimits}
\newcommand{\reg}{\mathop{\mathrm{reg}}\nolimits}

\newcommand{\tor}{\mathop{\mathrm{Tor}}\nolimits}
\renewcommand{\hom}{\mathop{\mathrm{Hom}}\nolimits}
\newcommand{\ext}{\mathop{\mathrm{Ext}}\nolimits}

\renewcommand{\*}{\bullet}

\title{On the Maximal Graded Shifts of Ideals and Modules}
\author[J. McCullough]{Jason McCullough}
\address{Department of Mathematics, Iowa State University, Ames, IA 50011}
\email{jmccullo@iastate.edu}
\begin{document}

\dedicatory{Dedicated to Professor Craig Huneke on the occasion of his sixty-fifth birthday.}

\begin{abstract}
We generalize a result of Eisenbud-Huneke-Ulrich on the maximal graded shifts of a module with prescribed annihilator and prove a linear regularity bound for ideals in a polynomial ring depending only on the first $p - c$ steps in the resolution, where $p = \pd(S/I)$ and $c = \codim(I)$. \end{abstract}

\maketitle

\section{Introduction}

Let $S = k[x_1,\ldots,x_n]$ be  a polynomial ring over a field $k$, viewed as a standard graded ring, and let $M$ be a finitely generated graded $S$-module.  Let $F_\*$ be the minimal graded free resolution of $M$ over $S$, so that $F_i = \oplus_{j} S(-j)^{\beta_{ij}(M)}$.  Following the notation in \cite{ES}, we denote the $i$th maximal and minimal graded  shifts as 
\begin{align*}
T_i(M) &= \max\{j\,|\,\tor_i^S(M,k)_j \neq 0\} = \max\{j\,|\,\beta_{ij}(M) \neq 0\},\\
t_i(M) &= \,\min\{j\,|\,\tor_i^S(M,k)_j \neq 0\} = \,\min\{j\,|\,\beta_{ij}(M) \neq 0\}.
\end{align*}
Other authors \cite{ACI2, HS, M} have used $t_i(M)$ to denote the maximal shifts but the above  notation seems more apt to the situation of dealing with both maximal and minimal shifts.  The maximal shifts are of primary interest due to their connection with Castelnuovo-Mumford regularity, that is 
\[\reg(M) = \max_{0 \le i \le \pd(M)}\{T_i(M) - i\}.\]
Fix $M = S/I$ to be a cyclic module, where $I$ is a homogeneous ideal of $S$.  A motivating question for this paper, which has received a lot of recent interest (see e.g. \cite{ACI1}, \cite{ACI2},  \cite{ES}, \cite{HS}) is the following:

\begin{qst} Which sequences of nonnegative integers occur as $\{T_i(S/I)\}$ for some homogeneous ideal $I$ in a polynomial ring $S$?
\end{qst}

There is a doubly exponential upper bound on $\reg(S/I)$ (and hence all of the $T_i(S/I)$) in terms of $T_1(S/I)$, the maximal degree of a minimal generator of $I$.  (See \cite{BM} and \cite{CS}.)  Examples derived from the Mayr-Meyer ideals  \cite{Ko} show that this upper bound is nearly optimal in that families of ideals generated by quadrics in $n$ variables can have first syzygies whose degree grows doubly exponentially in $n$.  Other examples of ideals with large regularity (e.g. \cite{BMNSSS}, \cite{BM}, \cite{BSt})  also have large degree first syzygies.  Later examples of ``designer ideals'' constructed by Ullery \cite{Ullery} showed that for any strictly increasing sequence of positive integers $2 \le a_1 < a_2 < \cdots < a_m$, there is an ideal $I$ in a polynomial ring $S$ such that $T_i(S/I) = a_i$ for all $1 \le i \le m$.  However, these ideals also satisfy $T_{m+i}(S/I) = a_m + i$ for $i > 0$ (i.e. there is a linear strand at the end of the resolution) and the resolutions tend to be rather long.  These examples motivate the idea that ideals whose resolutions show large degree jumps early in the resolution must be ``paid for'' by smaller degree jumps later in the resolution.  The first result quantifying this idea is the following weak convexity result due to Eisenbud-Huneke-Ulrich:

\begin{thm}[Eisenbud-Huneke-Ulrich {\cite[Corollary 4.1]{EHU}}]\label{EHUthm1} Let $I$ be a homogeneous ideal in $S = k[x_1,\ldots,x_n]$ such that $\dim(S/I) \le 1$.  Then
\[T_n(S/I) \le T_i(S/I) + T_{n-i}(S/I),\]
for all $0 \le i \le n$.
\end{thm}
\noindent Note that by replacing $n$ with $p:=\pd(S/I)$, the statement extends to ideals $I$ such that $\dim(S/I) - \depth(S/I) \le 1$, since we may always assume $k$ is infinite and then reduce both $S$ and $S/I$ by a sequence of linear forms regular on $S/I$ leaving us in the case when $\pd(S/I) = n$.  In particular, $T_p(S/I) \le T_i(S/I) + T_{p-i}(S/I)$ for all $0 \le i \le p = \pd(S/I)$ when $S/I$ is Cohen-Macaulay.  Previously the author asked \cite[Question 5.1]{M} whether the same weak convexity inequality held without the assumptions on $\depth(S/I)$ or $\dim(S/I)$.  The author knows of no counterexamples to this statement.  Note also that such a weak convexity inequality would imply a linear bound on $T_p(S/I)$ for \emph{any} ideal in terms of $T_1(S/I),\ldots,T_h(S/I)$, where $h := \left\lceil \frac{p}{2} \right\rceil$.  In \cite[Theorem 4.7]{M} the author proved a weaker polynomial bound on $\reg(S/I)$, and in particular on $T_p(S/I)$, in terms of $T_1(S/I),\ldots,T_h(S/I)$ without any restriction on $\depth(S/I)$ or $\dim(S/I)$ by means of the Boij-S\"oderberg decomposition of the Betti table of $S/I$.

Instead of discussing cyclic modules, one can more generally consider modules whose annihilator contains some fixed ideal.  In the same paper as above, Eisenbud-Huneke-Ulrich proved the following:

\begin{thm}[Eisenbud-Huneke-Ulrich {\cite[Corollary 4.2]{EHU}}]\label{EHUthm2} Let $M$ be a finitely-generated graded $S$-module of codimension $c$, projective dimension $p$, and with $\dim(M) - \depth(M) \le 1$.  Let $J$ be a homogeneous ideal contained in $\ann(M)$.  If $\depth(S/J) \ge \depth(M)$, then for $0 \le q \le \mathrm{codim}(J)$, 
\[T_{p}(M) \le T_{p-q}(M) + T_q(S/J).\]  
In particular:
\begin{enumerate}
\item[(a)] If $\ann(M)$ contains a regular sequence of degrees $d_1,\ldots,d_q$, then
\[T_{p}(M) \le T_{p - q}(M) + \sum_{i = 1}^q d_i.\]
\item [(b)] If $J$ is generated in degree $d$ with linear resolution, then
\[T_{p}(M) \le T_{p - q}(M) + d + q - 1.\]
\end{enumerate}
\end{thm}

We conjecture that Theorem~\ref{EHUthm2} also holds without the almost Cohen-Macaulay hypothesis.

\begin{conj}\label{mine} Let $M$ be a finitely-generated graded $S$-module of codimension $c$.  Let $J$ be a homogeneous ideal contained in $\ann(M)$ and set $p = \pd(M)$.  If $\depth(S/J) \ge \depth(M)$, then for $0 \le q \le \mathrm{codim}(J)$, 
\[T_{p}(M) \le T_{p-q}(M) + T_{q}(S/J).\]  
\end{conj}

Note that if Conjecture~\ref{mine} is true, then for any homogeneous ideal $I$ with $\pd(S/I) \le 2 \,{\mathrm{codim}}(I)$, Question 5.1 in \cite{M} has an affirmative answer.  

The purpose of this note is to give evidence in support of this conjecture when $S/J$ is Cohen-Macaulay.   Note that if true, the conjecture implies various linear bounds on $T_p(M)$ in terms of $T_0(M),T_1(M),\ldots,T_{p-c}(M)$ and $T_1(S/J),\ldots,T_c(S/J)$.  

As the main result of this note, we prove in Theorem~\ref{main} a weaker linear bound on $\reg(M)$, and thus also $T_p(M)$, in terms of the same data without restriction on $\depth(M)$ or $\dim(M)$.  In particular, when $M = S/I$ is a cyclic module, we get the following regularity bound in terms of only the first $\pd(S/I) - \codim(I)$ syzygy degrees. 

\begin{introcor} Let $I$ be an ideal of $S$ with $p = \pd(S/I)$ and $c = \codim(I)$.  Then
\[\reg(S/I) \le \max_{\substack{0 \le i \le p-c }}\left\{T_i(S/I) + (p-i)\, T_1(S/I) \right\} - p.  \] 
\end{introcor}

By adapting an argument of Herzog-Srinivasas \cite{HS}, we also prove in Theorem~\ref{codim1} the special case of Conjecture~\ref{mine} when $J$ is principal.  This recovers and strengthens results of Herzog-Srinivasan \cite{HS} and El Khoury-Srinivasan \cite{ES} as well as the author \cite{M}.  

In the next section we gather some preliminary results.  Section~\ref{mainresults} contains the main results and an example.

\section{Preliminary Results}\label{prelim}

We begin with some elementary facts about the maximal graded shifts in short exact sequences.

\begin{lem}\label{tiprop} Let $0 \to M' \to M \to M'' \to 0$ be a short exact sequence of graded $S$-modules.  Then for $i \ge 0$,
\begin{enumerate}
\item $T_i(M) \le \max\left\{T_i(M'), T_{i}(M'')\right\}$,
\item $T_i(M') \le \max\left\{T_i(M), T_{i+1}(M'')\right\}$,
\item $T_i(M'') \le \max\left\{T_i(M), T_{i-1}(M')\right\}$.
\end{enumerate}
\end{lem}

\begin{proof} All of the statements follow from the long exact sequence of $\tor$ stemming from tensoring the above short exact sequence with $k$.
\end{proof}

The following statement about the early maximal graded shifts seems to be well-known but we include a proof for completeness.  A similar result is proved in \cite[Lemma 1.2]{GL}.

\begin{prop}\label{Pinc} If $M$ be a finitely-generated graded $S$-module of codimension $c$, then $T_i(M) < T_{i+1}(M)$ for $0 \le i < c$.
\end{prop}

\begin{proof} Suppose to the contrary that $T_i(M) \ge T_{i+1}(M)$ with $0 \le i < c$.  Let $F_\*$ be the minimal graded free resolution of $M$ with differential maps $\partial_i:F_i \to F_{i-1}$.  Let $\alpha \in F_i - \m F_i$ be a homogeneous element of degree $T_i(M)$.  Since $F_\*$ is minimal and since $T_i(M) \ge T_{i+1}(M)$, after possibly a change of basis, the matrix representing $\partial_{i+1}$ has a row of zeros.  Now consider the dual $F^*_\* := \hom_S(F_\*,S)$, whose maps are represented by the transposes of the matrices representing the $\partial_i$;  we denote them by $\partial_i^*$.  Hence the matrix representing $\partial^*_{i+1}:F^*_i \to F^*_{i+1}$ has a column of zeros.  In other words, there is a basis element $\alpha^*$ (corresponding to the dual of $\alpha$) of $F^*_i$ which is in the kernel of $\partial^*_{i}$.  As $F_\ast$ was minimal, all entries of the matrix $\partial_{i+1}$, and hence of $\partial^*_{i+1}$, are contained in the maximal ideal.  Therefore $\alpha^*$ is a nonzero element in $\ker(\partial_{i}^*)/\im(\partial_{i+1}^*) = \ext^i_S(M,S)$ and in particular $\ext^i_S(M,S) \neq 0$.  This is impossible since $i < \mathrm{grade}(M) = \mathrm{codim}(M)$.
\end{proof}

It is also possible to prove the above result using the Boij-S\"oderberg decomposition  of the Betti table of $M$.  (See \cite{EisenbudSchreyer} and \cite{BS}.)   This decomposition involves multiples of Betti diagrams of pure Cohen-Macaulay $S$-modules of codimension at least $\mathrm{codim}(M)$.  The associated degree sequences are strictly increasing thus forcing the inequality in the above proposition.

\begin{rmk} We note that in general it is possible that $t_i(M) \ge t_{i+1}(M)$ for $i \ge \mathrm{codim}(M)$, even if $M$ is cyclic.  A Macaulay2 \cite{Mac2} computation shows that if 
\[S = K[x_1,x_2,y_1,y_2,z_1,z_2,z_3,z_4]\]
 and 
\[I = \left(x_1^6,y_1^6,x_1^2x_2^4+y_1^2y_2^4+x_1y_1(x_2^3z_1+x_2^2y_2z_2+x_2y_2^2z_3+y_2^3z_4)\right),\]
 then $T_7(S/I) = 38$ while  $T_8(S/I) = 34$.  Note that $\mathrm{codim}(S/I) = 2$.
 
 One can also create more extreme examples by taking direct sums of cyclic modules of different codimensions.  For example, if $S = K[x,y,z]$ and $M = S/(x^n,y^n) \oplus S/(x,y,z)$, then $\mathrm{codim}(M) = 2$, $T_2(M) = 2n$ and $T_3(M) = 3$.  
\end{rmk}

An easy consequence of Theorem~\ref{EHUthm2} and Proposition~\ref{Pinc} is the following statement regarding the regularity of Cohen-Macaulay modules.

\begin{cor}\label{EHUcor} If $M$ is a Cohen-Macaulay module of codimension $c$ and $\ann(M)$ contains a homogeneous Cohen-Macaulay ideal $J$ also of codimension $c$, then 
\[\reg(M) \le T_0(M) + T_c(S/J) - c.\]
In particular, if $\ann(M)$ contains a regular sequence of degrees $d_1,\ldots,d_c$, then 
\[\reg(M) \le T_0(M) + \left(\sum_{j = 1}^c d_j\right) - c.\]
\end{cor}

\begin{proof} By Proposition~\ref{Pinc}, we have $\reg(M) = T_c(M) - c$.  Since $\pd(S/J) = \pd(M) = c$, the rest follows from Theorem~\ref{EHUthm2}.
\end{proof}







So if $M$ is Cohen-Macaulay and $\ann(M)$ contains a regular sequence of length $\codim(M)$, then one needs only the maximal degree of a minimal generator (i.e. the $0$th step in the free resolution of $M$) to get an effective bound on $\reg(M)$.  In particular, the previous corollary shows that among Cohen-Macaulay ideals of fixed codimension and maximal generating degree, complete intersections have the maximum regularity.  The goal of this paper is to show that for arbitrary modules $M$, one need only compute the first $\pd(M) - \mathrm{codim}(M)$ steps of the resolution of $M$ to get an effective bound on the regularity of $M$.




\section{Main Results}\label{mainresults}

First we prove the previously mentioned special case of Conjecture~\ref{mine}.  The proof, which we include for completeness, is adapted from that of Herzog-Srinivasan \cite[Proposition 2]{HS} 

\begin{thm}\label{codim1} Let $M$ be a graded $S$-module of projective dimension $p$.  Suppose $f \in \ann(M)$ and set $d = \deg(f)$.  Then
\[T_p(M) \le T_{p-1}(M) + d = T_{p-1}(M) + T_1(S/(f)).\]
In particular, for any homogeneous ideal $I$ of $S$, if $p = \pd(S/I)$ then
\[T_p(S/I) \le T_{p-1}(S/I) + t_1(S/I).\]
\end{thm}

\begin{proof}
Let $F_\*$ denote the graded minimal free resolution of $M$ with differential maps $\partial_i:F_i \to F_{i-1}$.   Let $F^\ast_\*= \hom_S(F_\*,S)$ be the dual complex with differential maps $\partial_i^\ast:F_{i-1}^\ast \to F_i^\ast$.  Let $g_1,\ldots,g_r$ denote a homogeneous basis of $F_{p-1}$ and let $h_1,\ldots,h_s$ denote a homogeneous basis of $F_p$.  Let $g_1^\ast,\ldots,g_r^\ast$ and $h_1^\ast,\ldots,h_s^\ast$ be the corresponding dual bases of $F_{p-1}^\ast$ and $F_p^\ast$, respectively.  The map $\partial_p:F_p \to F_{p-1}$ is represented by a matrix $(c_{ij})$ with respect to these basis, and $\partial_p^\ast$ is represented by $(c_{ij})^\mathsf{T}$.  Since $F_{p+1}^\ast = 0$,  $\partial_{p+1}^\ast(h_i^\ast) = 0$ for all $i$.  Since $F_\*$ is minimal, for every $i$, $h_i^\ast$ is a nonzero minimal generator of $H^{p}(F_\*^\ast) = \ext_S^p(M,S)$.  Since $f \in \ann(M) \subseteq \ann\left( \ext_S^p(M,S)\right)$, $f\cdot h_i^\ast \in \im(\partial_{p}^\ast)$ for all $i$.

Now fix $i$.  Write $\partial_p(h_i) = \sum_{j = 1}^r c_{ij} g_j$, and notice that $c_{ij} \neq 0$ for some $j$ since $\partial_p$ is injective.  Note also that $\partial_p^\ast(g_j^\ast) = \sum_{i = 1}^s c_{ij} h_i^\ast$ for all $j$.  Therefore, 
since $f \cdot h_i^\ast \in \im(\partial_p^\ast)$, 
\[f \cdot h_i^\ast = \sum_{j = 1}^r b_j \partial_p^\ast(d_j^\ast) = \sum_{j = 1}^r b_j \sum_{i = 1}^s c_{ij} h_i^\ast,\]
for some $b_1,\ldots,b_r \in S$.
Then since at least one of the $c_{ij} \neq 0$ we have
\[\deg(f) + \deg(h_i^\ast) = \deg(b_j) + \deg(c_{ij}) + \deg(h_i^\ast),\]
and hence
\[d = \deg(f) = \deg(b_j) + \deg(c_{ij}) \ge \deg(c_{ij}).\]
Since $\partial_p(h_i) = \sum_{j = 1}^r c_{ij} g_j$ we have
\[\deg(h_i) = \deg(c_{ij}) + \deg(g_j) \le d + T_{p-1}(M)\]
for all $i$.  Therefore $T_p(M) \le d + T_{p-1}(M).$

The second statement follows by selecting $f \in I$ of degree $t_1(S/I)$.  Then since $f \in \ann(S/I)$ we have
\[T_p(S/I) \le T_{p-1}(S/I) + T_1(S/(f)) = T_{p-1}(S/I) + t_1(S/I).\]
\end{proof}

Note that $t_1(S/I)$ is just the minimal degree of a minimal generator of $I$.  Theorem~\ref{codim1} recovers and strengthens \cite[Corollary 3]{HS}, \cite[Theorem 2.1]{ES} and \cite[Theorem 4.4]{M}.  The theorem implies the following bound on regularity in terms of the first $p-1$ syzygies:

\begin{cor}\label{Ccodim1} Let $M$ be a graded $S$-module with $\pd(M) = p$.  Suppose $f \in \ann(M)$ and set $d = \deg(f)$.  
\[\reg(M) \le \max\left\{\max_{0 \le i \le p-1} \{T_i(M) - i\}, T_{p-1}(M)+d-p\right\}.\]

\end{cor}

\begin{proof} By definition $\reg(M) \le \max_{0 \le i \le p} \{T_i(M) - i\}$.  By the previous theorem $T_p(M) - p \le T_{p-1} + d - p$.
\end{proof}

Next we prove a weaker bound on $\reg(M)$ but one that is still only dependent on the first $\pd(M) - \codim(M)$ steps of the resolution.  First we make a  notational definition for convenience.  Let $\N$ denote the set of non-negative integers.  Let $\a = (a_1,\ldots,a_q) \in \N^{q}$.   Define $|\a| = \sum_{i = 1}^q a_i  (i+1)$.  For instance $|(1,2,0,1)| = 1 \cdot 2 + 2 \cdot 3 + 0\cdot 4 + 1 \cdot 5 = 13$.

We now consider arbitrary modules $M$ with a fixed Cohen-Macaulay ideal $J \subseteq \ann(M)$.  The assumption that $J$ is Cohen-Macaulay is necessary to make the argument work.  
  In the case when $J$ is a complete intersection of linear forms, then $\overline{S} = S/J$ is a polynomial ring and the regularity of $M$ may be computed over $\overline{S}$.  Hence we focus on the case when $J$ has at least one generator of at least degree $2$.  Our main result is the following:

\begin{thm}\label{regthm} Let $M$ be a finitely generated $S$-module.  Suppose that $\ann(M)$ contains a homogeneous ideal $J$ such that $S/J$ is Cohen-Macaulay with $c = \codim(J) = \pd(S/J)$ and $T_1(S/J) \ge 2$.  Set $p = \pd(M)$.
Then 
\[\reg(M) \le \max_{\substack{0 \le i \le p-c\\\a \in \N^c\\|\a| \le p - c - i }}\left\{T_i(M) + \sum_{j = 1}^c a_j (T_j(S/J))\right\}  + T_c(S/J) - p.  \]
\end{thm}

\begin{proof}
Set  $R = S/J$.  Since $M$ is an $R$-module, $p = \pd_S(M) \ge \mathrm{codim}(M) \ge \mathrm{codim}(R) = c$.  Note that if $c = 1$, then the statement follows from Corollary~\ref{Ccodim1}. 

We proceed by induction on $p$.

  If $p = c$, then $M$ is Cohen-Macaulay.  Hence $\reg(M) \le T_0(M) + T_c(S/J) - c$ by Corollary~\ref{EHUcor}, which proves the base case, since $\underline{0} = (0,0,\ldots,0) \in \N^c$ is the only vector $\a \in \N^c$ such that $|\a| \le 0$.
  
  If $p > c$, then consider the short exact sequence of $R$-modules (and thus also $S$-modules)
\[0 \to K \to F \to M \to 0.\]
Here $F = \oplus_j R(-d_j)^{\beta_{0,j}(M)}$ is the $0$th step in a minimal $R$-free resolution of $M$, and $K = \ker(F \onto M)$.  (In other words, $K$ is the module of first syzygies of $M$ over $R$.)   Since $F$ is free over $R$, an $S$-free resolution of $F$ is merely an $S$-free resolution of $S/J$ tensored by $\oplus_j S(-d_j)^{\beta_{0,j}(M)}$ over $S$.  In particular, $\pd_S(F) = \pd_S(R) = c < p = \pd_S(M)$ and $T_i(F) = T_0(F) + T_i(S/J)$ for $0 \le i \le c$.  Moreover, by our choice of $F$ we have $T_0(F) = T_0(M)$.  
 It follows that $\pd_S(K) = \pd_S(M) - 1 = p - 1$.  Since $\pd_S(S/J) = c < p = \pd_S(M)$, $T_p(M) \le T_{p-1}(K)$.

  If $p = c+1$, then $\dim(M) - \depth(M) = \pd(M) - \codim(M) = 1$.  The only vector $\a \in \N^c$ with $|\a| \le 1$ is again $\underline{0}$.  Again by Corollary~\ref{EHUcor}, we have $T_p(M) \le T_1(M) + T_{p-1}(S/J)$.  It follows from Proposition~\ref{Pinc} that $\reg(M) = \max\{T_p(M) - p, T_{p-1}(M) - (p-1)\}$.  Thus it suffices to show that $T_{p-1}(M) \le T_1(M) + T_{p-1}(S/J) - 1$.  This follows by noting
  \begin{align*}
  T_{p-1}(M) &\le \max\{T_{p-1}(F), T_{p-2}(K)\}\\
  &\le \max\{T_0(F) + T_{p-1}(S/J), T_{p-1}(K) - 1\}  \\
  &\le \max\{T_0(M) + T_{p-1}(S/J), T_0(K) + T_{p-1}(S/J) - 1\}\\
  &\le T_1(M)  + T_{p-1}(S/J) - 1.
  \end{align*}
  The first inequality is just Lemma~\ref{tiprop}; the second inequality follows from the discussion of $T_i(F)$ and from Proposition~\ref{Pinc}; the third inequality follows from Theorem~\ref{EHUthm2} applied to $K$; the final inequality is again Proposition~\ref{Pinc} and Lemma~\ref{tiprop}.


 If $p > c+1$, then since $J \subset \ann(K)$ and since $\pd(K) < p$, by the inductive hypothesis we have
 \[\reg(K) \le T_i(K) + \sum_{j = 1}^c a_j T_j(S/J)  + T_c(S/J) - (p-1),  \]
for some $0 \le i \le (p-1)-c$ and some $\a \in \N^c$ with $|\a| \le (p-1)-c-i$.  Using the short exact sequence above we have 
\[\reg(M) \le \max\{\reg(F), \reg(K) - 1\}.\]
Since $F$ is Cohen-Macaulay,
\begin{align*}
\reg(F) &= T_c(F) - c \\
&= T_0(F) + T_c(S/J) - c \\
&= T_0(M) + T_c(S/J) - c.\\
\end{align*}
If $p-c$ is even, say $p-c = 2m$ for some nonnegative integer $m$, then we set $\a = (m,0,0,\ldots)$ and note that $|\a| = 2m = p-c$. 
If $p-c$ is odd, since $p > c+1$, $p$ must be at least $3$.  Hence we can write $p-c = 2m + 3$ for some nonnegative integer $m$.  Noting that $c \ge 2$ by assumption, we set $\a = (m,1,0,0,\ldots)$ and see that again $|\a| = 2m + 3 = p-c$.  Since $T_1(S/J) \ge 2$  we have $T_j(S/J) \ge j+1$ for $1 \le j \le c$.  Therefore $\sum_{j=1}^c a_j T_j(S/J) \ge \sum_{j=1}^c a_j(j+1) = |\a| = p-c$;  thus, in either case we have
\[\reg(F) \le T_0(M) + \sum_{j=1}^c a_j T_j(S/J) +T_c(S/J) - p.\]
Hence we may assume that 
\[\reg(M) \le \reg(K) - 1 = T_i(K) + \sum_{j = 1}^c a_j T_j(S/J)  + T_c(S/J) - p,\]
for some $0 \le i \le c$ and some $\a \in \N^c$ with $|\a| \le (p-1) - c - i$.
By Lemma~\ref{tiprop}, $T_i(K) \le T_{i+1}(M)$ or $T_i(K) \le T_i(F)$.  
In the former case we have
\[\reg(M) \le T_{i+1}(M)  + \sum_{j = 1}^c a_j T_j(S/J)  + T_c(S/J) - p .\]
In the latter case, since $T_i(F) = T_0(M) + T_i(S/J)$, we have
\begin{align*}
\reg(M) 
&\le T_0(M) + T_i(S/J) + \sum_{j = 1}^c a_j T_j(S/J)  + T_c(S/J) - p\\
&= T_0(M) + \sum_{j = 1}^c a_j' T_j(S/J) + T_c(S/J) - p,\\
\end{align*}
where 
\[a_j' = \begin{cases} a_j & \text{ if } j \neq i\\
a_j + 1 & \text{ if } j = i.
\end{cases}\]
Finaly we note that 
\begin{align*}
|\a'| &= \sum_{j = 1}^c a_j'(j+1) \\
&= (i+1) + \sum_{j=1}^c a_j(j+1) \\
&= (i+1) + |\a| \\
&\le (i + 1) + (p-1) - c - i\\
& = p - c 
\end{align*}
as desired.
\end{proof}

As stated, the bound may be a bit hard to interpret.  We include a specific instance of its application below.

\begin{ex}
Suppose $\pd(M) = 7$ and $\codim(M) = 3$.  Further suppose that $J = (x,y,z)^2 \subset \ann(M)$.  Then $T_i(S/J) = i + 1$ for $1 \le i \le 3$.  We consider vectors $\a  \in \N^3$ with $|\a| = p - c \le 4$.  No vectors satisfy $|\a| = 1$; only the vectors $(1,0,0)$ and $(0,1,0)$ satisfy $|\a| = 2$ and $3$, respectively.  The only such vectors $\a$ with $|\a| = 4$ are $(2,0,0)$ and $(0,0,1)$.  Therefore, the previous theorem implies that
\begin{align*}
\reg(M) &\le \max\{T_4(M) + T_3(S/J) - 7, T_2(M) + T_1(S/J) + T_3(S/J) - 7, \\
&\phantom{\le \max T}T_1(M) + T_2(S/J) + T_3(S/J) - 7,  T_0(M) + 2T_3(S/J) - 7, \\
&\phantom{\le \max T}T_0(M) + 2T_1(S/J) + T_3(S/J) - 7 \}\\
& = \max\{T_4(M) - 3, T_2(M) - 1, T_1(M) , T_0(M) + 1\}\\
& = \max\{T_4(M) - 3, T_2(M) - 1\},
\end{align*}
where the second equality follows since $\codim(M) = 3$ and hence $T_0(M) < T_1(M) < T_2(M) < T_3(M)$ by Proposition~\ref{Pinc}.  
\end{ex}

As a corollary to our main theorem, we get a weak upper bound on $T_p(M)$ in terms of the same data.

\begin{cor}\label{main} Let $M$ be a finitely generated $S$-module.  Suppose that $\ann(M)$ contains a homogeneous ideal $J$ such that $S/J$ is Cohen-Macaulay with $c = \codim(J) = \pd(S/J)$.  Set $p = \pd(M)$.  Then
\[T_p(M) \le \max_{\substack{0 \le i \le p-c\\\a \in \N^c\\|\a| \le p - c - i }}\left\{T_i(M) + \sum_{j = 1}^c a_j T_j(S/J)\right\} + T_c(S/J).\]

\begin{proof}
Since $T_p(M) \le \reg(M) + p$, the inequality follows from Theorem~\ref{regthm}.
\end{proof}


\end{cor}

Finally we state a few special cases of the main theorem in the cyclic case.

\begin{cor} Let $I$ be an ideal of $S$ with $p = \pd(S/I)$ and $c = \codim(I)$.  If $I$ contains a regular sequence of forms of common degree $d$ then
\[\reg(S/I) \le \max_{\substack{0 \le i \le p-c }}\left\{T_i(S/I) + (p-i)\, d \right\}-p.  \]
\end{cor}

\begin{proof} Let $f_1,\ldots,f_c$ be a regular sequence of forms of degree $d$ contained in $I$.  Then $J = (f_1,\ldots,f_c) \subset \ann(S/I)$ is a Cohen-Macualay ideal and, since $S/J$ is resolved by a Koszul complex on $f_1,\ldots,f_c$, we have $T_i(S/J) = di$ for $0 \le i \le c$.   By Theorem~\ref{regthm} we have for some $0 \le i \le p-c$ and some $\a \in \N^c$ with $|\a| \le p-c-i$ that
\begin{align*}
\reg(S/I) &\le  T_i(S/I) + \sum_{j = 1}^c a_j (T_j(S/J))  + T_c(S/J) - p\\
&= T_i(S/I) + \sum_{j = 1}^c (a_j dj)  + dc - p\\
&\le T_i(S/I) + d\left(\sum_{j = 1}^c (a_j (j+1))\right)  + dc - p\\
&\le T_i(S/I) + d(p-c-i)  + dc - p\\
&= T_i(S/I) + d(p-i)  - p\\
\end{align*}

\end{proof}

Thus we get a bound on the regularity of any cyclic module $S/I$ purely in terms of the first $\pd(S/I) - \codim(I)$ maximal graded shifts:

\begin{cor}\label{maincor} Let $I$ be an ideal of $S$ with $p = \pd(S/I)$ and $c = \codim(I)$.  Then
\[\reg(S/I) \le \max_{\substack{0 \le i \le p-c }}\left\{T_i(S/I) + (p-i)\, T_1(S/I) \right\} - p.  \]
\end{cor}

\begin{proof} This follows from the previous corollary since $I$ contains a regular sequence of forms of length $c$ and of degree at most $T_1(S/I)$.
\end{proof}

This upper bound is likely not tight for all values of $p$ and $c$, but it seems close to a best possible result in the following sense:

Fix a value of $c$.  A rephrasing of the Ullery's result implies that for $p \gg 0$ there is no bound on $\reg(S/I)$ purely in terms of $T_1(S/I),\ldots,T_{p-c-1}(S/I)$, where $p = \pd(S/I)$ and $c = \codim(I)$.  Thus the result above, which gives a linear bound on $\reg(S/I)$ in terms of $T_1(S/I),\ldots,T_{p-c}(S/I)$ and $\pd(S/I)$ uses minimal data in some sense.  Where Ullery's result shows that almost any sequences of increasing positive integers appear as the initial maximal graded shifts of some cyclic module or ideal, our result restricts what can happen at the end of the resolution.

\section*{Acknowledgements}

The author thanks Lance Miller, Irena Peeva and Hema Srinivasan for useful feedback on an earlier draft of this paper.


\begin{thebibliography}{cccc}

\bibitem{ACI1}{L. Avramov, A. Conca and S. Iyengar, \emph{Free resolutions over commutative Koszul algebras}, Math. Res. Lett. {\bf 17} (2010), no. 2, 197--210.  }

\bibitem{ACI2}{L. Avramov, A. Conca and S. Iyengar, \emph{Subadditivity of syzygies of Koszul algebras}. Math. Ann. {\bf 361} (2015), no. 1-2, 511--534. }

\bibitem{BMNSSS}{J. Beder, J. McCullough, L. N\'unez-Betancourt, Luis, A. Seceleanu, B. Snapp, and B. Stone, \emph{Ideals with larger projective dimension and regularity}, J. Symbolic Comput. {\bf 46} (2011), no. 10, 1105-1113.}

\bibitem{BM} 
D. Bayer and D. Mumford:  \emph{What can be computed in Algebraic Geometry?}, Computational Algebraic Geometry and Commutative Algebra, Symposia Mathematica, Volume XXXIV, Cambridge University Press, Cambridge, 1993, 1--48.

\bibitem{BSt} 
D. Bayer and M. Stillman: \emph{On the complexity of computing syzygies. Computational aspects of commutative algebra},  J. Symbolic Comput. {\bf 6} (1988), 135--147.


\bibitem{BS}{M. Boij and J. S\"oderberg, \emph{Betti numbers of graded modules and the multiplicity conjecture in the non-Cohen-Macaulay case},
Algebra Number Theory {\bf 6} (2012), no. 3, 437--454.} 

\bibitem{CS} 
G. Caviglia and E. Sbarra: \emph{Characteristic-free bounds for the Castelnuovo-Mumford regularity}, Compos. Math. {\bf 141} (2005),  1365--1373. 


\bibitem{Eisenbud3}{D. Eisenbud, \emph{Commutative algebra. With a view toward algebraic geometry}. Graduate Texts in Mathematics, 150. Springer-Verlag, New York (1995).}

\bibitem{EHU}{D. Eisenbud, C. Huneke and B. Ulrich, \emph{The regularity of {T}or and graded {B}etti numbers}, Amer. J. Math. {\bf 128} (2006), 573--605.}


\bibitem{EisenbudSchreyer}{D. Eisenbud and F.-O. Schreyer, \emph{Betti numbers of graded modules and cohomology of vector bundles}, 
J. Amer. Math. Soc. {\bf 22} (2009), no. 3, 859--888.}

\bibitem{ES}{S. El Khouri and H. Srinivasan, \emph{A note on the subadditivity of syzygies}, preprint: 	arXiv:1602.02116.}

\bibitem{GL}{M. Green and R. Lazarsfeld, \emph{Some results on the syzygies of finite sets and algebraic curves}, Compositio Math. {\bf 67} (1988), no. 3, 301--314.}

\bibitem{Ha}{J. Harris, \emph{Algebraic geometry. A first course.} Graduate Texts in Mathematics {\bf 133}. Springer-Verlag, New York (1992).}

\bibitem{HS}{J. Herzog and H. Srinivasan, \emph{On the subadditivity problem for maximal shifts in free resolutions}, Commutative algebra and noncommutative algebraic geometry. Vol. II, 
Math. Sci. Res. Inst. Publ. {\bf 68}, Cambridge Univ. Press, New York (2015), 245--249.}

  \bibitem{Ko}   
J. Koh: \emph{Ideals generated by quadrics exhibiting double exponential degrees,} J. Algebra {\bf 200} (1998),  225--245.


\bibitem{MM}
E. Mayr and A. Meyer: \emph{The complexity of the word
problem for commutative semigroups and polynomial ideals}, Adv.
in Math. {\bf 46} (1982), 305--329.



\bibitem{M}{J. McCullough, \emph{A polynomial bound on the regularity of an ideal in terms of half of the syzygies}, Math. Res. Lett. {\bf 19} (2012), no. 3, 555--565.} 

\bibitem{Mac2}{ Macaulay 2, a software system for research in algebraic
geometry. Available at http://www.math.uiuc.edu/Macaulay2/}

\bibitem{Ullery}{B. Ullery, \emph{Designer ideals with high Castelnuovo-Mumford regularity}, 
Math. Res. Lett. 21 (2014), no. 5, 1215--1225. }

\end{thebibliography}

\end{document}